\documentclass[11pt, leqno]{amsart}

\usepackage{amssymb,amsfonts}
\usepackage{enumitem}
\usepackage{hyperref}
\usepackage[margin=2.80cm]{geometry}
\usepackage[dvipsnames]{xcolor}

\newtheorem{thm}{Theorem}[section]
\newtheorem{cor}[thm]{Corollary}
\newtheorem{prop}[thm]{Proposition}
\newtheorem{lem}[thm]{Lemma}

\theoremstyle{definition}

\theoremstyle{remark}
\newtheorem{rem}[thm]{Remark}

\makeatletter
\let\c@equation\c@thm
\makeatother
\numberwithin{equation}{section}

\title[Planarity estimates and classification at the Andrews--Baker threshold]{Planarity estimates and classification of ancient flows at the Andrews--Baker threshold}

\thanks{The author is partially supported by INdAM-GNSAGA}



\author[Francesco Ruatta]{Francesco Ruatta}
\address[Francesco Ruatta]{Dipartimento di Scienze Matematiche "Giuseppe Luigi Lagrange", Politecnico di Torino, Corso Duca degli Abruzzi, 24, Torino, Italy, I-10129}
\email{francesco.ruatta@polito.it}

\hypersetup{
    colorlinks=true,
    linkcolor=RawSienna,
    citecolor=RawSienna,
    filecolor=magenta,
    urlcolor=RawSienna,
    pdftitle={Planarity estimates and classification of ancient flows at the Andrews--Baker threshold},
}

\allowdisplaybreaks

\begin{document}


\begin{abstract}
    We study space-time proper ancient solutions of the mean curvature flow of $n$-dimensional submanifolds of $\mathbb R^{n+p}$ that are uniformly quadratically 
pinched up to the Andrews--Baker threshold $\frac{4}{3n}$. We prove codimension reduction and convexity: every such solution is either a static affine $n$-plane, or a flow of hypersurfaces of a fixed affine $(n+1)$-dimensional subspace whose time slices bound convex bodies. By results of Bourni, Langford and Lynch and of Bamler and Lai, the flow is then noncollapsed, and it is a family of shrinking spheres or cylinders, an ancient oval, a bowl soliton or a flying wing, possibly times a Euclidean factor. Space-time properness replaces the finite entropy and polynomial curvature growth assumptions of Lee, Naff and Zhu. The main tool is a parabolically scale-invariant interior planarity estimate for $n\geq5$, obtained from Naff's gradient estimates through a localization argument in the spirit of Ecker and Huisken. In dimensions five and six, where these estimates do not give the gradient terms a sign, we bound them from below as in a previous work. The same localization, applied to the gradient estimate of Lynch and Nguyen and to a second order analogue, gives the derivative bounds from which convexity follows. For closed flows, we extend Lynch's planarity estimate to a larger range of exponents.
\end{abstract}

\maketitle


\section{Introduction}

Given $n\geq 2$ and $p\geq 1$, a one-parameter family of immersed and connected submanifolds without boundary $x(\cdot,t):M^n\to \mathbb R^{n+p}$ moves by mean curvature flow if $\partial_tx=\mathbf H(\cdot,t)$, where the mean curvature vector $\mathbf H(\cdot,t)$ is defined as the trace of the second fundamental form $B$ with respect to the induced metric $g$. A solution is called ancient if it is defined on the time interval $(-\infty,T)$, where $T\leq \infty$. 
Ancient solutions arise as blow-up limits at singularities, which motivates their classification. In codimension one, this program is well developed for convex solutions. Bourni, Langford and Lynch proved that a convex ancient solution with $H>0$ is noncollapsed if and only if it sweeps out the whole space, and that otherwise its blow-down is a hyperplane of multiplicity two; see \cite[Theorems $1.2$ and $1.3$]{BourniLangfordLynch_Noncollapsing}. Noncollapsed, uniformly $2$-convex ancient solutions have been classified in the closed case in \cite{ASD_ClassificationClosed2Convex} by Angenent, Daskalopoulos and \v{S}e\v{s}um, and in the strictly-convex, noncompact case in \cite{BrendleChoi_Classification1,BrendleChoi_Classification2} by Brendle and Choi. 
For convex flows of dimension $n\geq 3$, the noncollapsing assumption in these two results can be dropped; see \cite[Corollary $1.7$]{BourniLangfordLynch_Noncollapsing}.
The classification of ancient flows whose blow-down is the shrinking cylinder $\mathbb R\times S^{n-1}$ is due to Choi, Haslhofer, Hershkovits and White; see \cite{CHH_ClassificationCylindrical1,CHHW_ClassificationCylindrical2}. Recently, Bamler and Lai extended this last classification to ancient flows whose blow-down is the generalized shrinking cylinder $\mathbb R^k\times S^{n-k}$, with $0\leq k\leq n-1$; see \cite{BamlerLai_Classification}.

A submanifold is \emph{$c$-(quadratically) pinched} (resp. \emph{uniformly $c$-(quadratically) pinched}) if $|B|^2\leq c|\mathbf H|^2$ (resp. if $|B|^2\leq (c-\varepsilon_0)|\mathbf H|^2$ for some $\varepsilon_0>0$). This property is important mainly for two reasons. First, the pinching condition $|B|^2+a<c|\mathbf H|^2$, with $a>0$, is preserved along closed flows provided that $c\leq\frac{4}{3n}$, as shown in \cite{AndrewsBaker_MCFPinchedSubmanifolds,Lynch_PlanarityEstimate}. Second, in codimension one, every uniformly $\frac{1}{n-k}$-pinched hypersurface (with $1\leq k\leq n-1$) whose unit normal is chosen so that the scalar mean curvature $H$ is positive is uniformly $k$-convex; see \cite[Lemma 5.1]{HuiskenSinestrari_MCFConvexAncientSolutions}. 
As a result, a uniformly $\frac{4}{3n}$-pinched hypersurface with $H>0$ is uniformly $(n-1)$-convex. On the other hand, the cylinder $\mathbb R^k\times S^{n-k}$, with $0\leq k\leq n-1$, satisfies $|B|^2=\frac{1}{n-k}|\mathbf H|^2$, and hence it is uniformly $\frac{4}{3n}$-pinched if and only if $4k<n$. These two facts determine the models appearing in Theorem \ref{thm:main_theorem_ancient} below; see Remark \ref{rem:classification_dimension}.

In the compact, codimension one case, Huisken and Sinestrari \cite{HuiskenSinestrari_MCFConvexAncientSolutions} gave various conditions
ensuring that a closed convex ancient solution is a family of shrinking spheres, for instance $h\geq \varepsilon Hg$ for some $\varepsilon>0$. This rigidity was later extended to arbitrary codimension by Lynch and Nguyen \cite{LynchNguyen_PinchedAncientSolutionsHighCodimMCF}, who showed that closed, uniformly $\min\left\{\frac{4}{3n},\frac{1}{n-1}\right\}$-pinched ancient solutions are shrinking spheres.

In general, under quadratic pinching one expects \emph{codimension reduction}. For $n\geq 5$, this was established by Naff \cite{Naff_Planarity} through the so-called \emph{planarity estimate}, recalled below, under the pinching constant
\begin{equation*}
    c_n^{\mathrm{Naff}}\doteq\min\left\{\frac{3(n+1)}{2n(n+2)}, \frac{4}{3n}\right\},
\end{equation*}
which is smaller than $\frac{4}{3n}$ precisely when $n\leq 6$. Naff showed that, as a consequence of this estimate, blow-up limits at the first singular time of closed flows satisfying $|B|^2<c_n^{\mathrm{Naff}} |\mathbf H|^2$ are codimension one; see \cite[Corollary $2.6$]{Naff_Planarity}. Bourni, Langford and Lynch proved in \cite[Corollary $1.6$]{BourniLangfordLynch_Noncollapsing} that, up to a fixed rotation, the time slices of the nonflat ones bound convex domains in $\mathbb R^{n+1}$ whose inscribed radius is at least $|\mathbf H|^{-1}$. If, in addition, $|B|^2<\frac{1}{n-2}|\mathbf H|^2$, Naff \cite[Corollary $1.4$]{Naff_SingularityModels} showed that for $n\geq 5$ these limits are shrinking spheres, shrinking cylinders or bowl solitons.

In the noncompact case, Lee, Naff and Zhu \cite{LNZ_Planarity} proved a parabolic planarity estimate for complete flows with bounded curvature under uniform 
$c_n^{\mathrm{Naff}}$-pinching, and consequently obtained codimension reduction for ancient solutions with the same pinching, bounded entropy and polynomial curvature growth. They also showed that, in codimension one, ancient mean-convex flows with finite entropy and curvature of polynomial growth that satisfy $|B|\leq L|\mathbf H|$, for some $L<\infty$, are convex; see \cite[Theorem $1.3$]{LNZ_Planarity}. Combining these results with
\cite{BourniLangfordLynch_Noncollapsing}, they showed that uniformly $c_n^{\mathrm{Naff}}$-pinched ancient solutions with bounded entropy and polynomial curvature growth are convex and noncollapsed. Using \cite{ASD_ClassificationClosed2Convex,BrendleChoi_Classification1,BrendleChoi_Classification2}, they deduced that, for $n\geq 5$, those which are uniformly $\min\left\{c_n^{\mathrm{Naff}},\frac{1}{n-2}\right\}$-pinched are static planes, shrinking spheres or cylinders, bowl solitons or ancient ovals; cf. \cite[Corollary $1.5(1)$]{LNZ_Planarity}.

The goal of this article is to remove the entropy and curvature growth assumptions from this picture, replacing them with space-time properness, and to reach the Andrews--Baker threshold $\frac{4}{3n}$ in every dimension. The main tool is a scale-invariant planarity estimate for space-time proper and uniformly $\frac{4}{3n}$-pinched flows in dimensions $n\geq 5$; see Theorem \ref{thm:main_result_planarity_estimate} below.

A mean curvature flow $x:M^n\times I\to\mathbb R^{n+p}$ is said to be \emph{space-time proper} if the map $M\times I\ni(q,t)\mapsto(x(q,t),t)\in\mathbb R^{n+p}\times I$ is 
proper, which makes every time slice $M_t\doteq x(M,t)$ complete. In this setting, we formulate the main result on ancient solutions.

\begin{thm}[Codimension reduction and classification of ancient solutions]\label{thm:main_theorem_ancient}
    Let $n\geq 2$, $p\geq 1$ and $x:M^n\times (-\infty,T)\to\mathbb R^{n+p}$ be a space-time proper and ancient mean curvature flow of a connected manifold $M$ without boundary. Assume that $|B|^2\leq\left(\frac{4}{3n}-\varepsilon_0\right)|\mathbf H|^2$ for some $\varepsilon_0>0$. Then either $M_t$ is a static, flat copy of $\mathbb R^n$ in $\mathbb R^{n+p}$, or the following holds. 
    \begin{enumerate}[label=(\alph*), ref=\thethm(\alph*)]
        \item\label{item:codim_reduction} $|\mathbf H|>0$ on $M\times(-\infty,T)$, and the flow is contained in a fixed, affine $(n+1)$-dimensional subspace $\mathcal A\subseteq \mathbb R^{n+p}$.
        \item\label{item:embedding_noncollapsness} Each $x(\cdot,t)$ is an embedding onto the boundary of a convex body $\Omega_t\subseteq \mathcal A$ and the flow is noncollapsed, i.e. the inscribed radius of $\Omega_t$ is at least $|\mathbf H|^{-1}$ at every point of $M_t=\partial\Omega_t$.
        \item\label{item:classification} The flow has blow-down  $\mathbb R^k\times S^{n-k}$, with $4k<n$, and is the restriction to $(-\infty,T)$ of one of the following:
        \begin{enumerate}[label=(\roman*)]
            \item a family of shrinking round spheres or cylinders $\mathbb R^k\times S^{n-k}$;
            \item an ancient oval, possibly times a Euclidean factor;
            \item a bowl soliton, possibly times a Euclidean factor;
            \item a flying wing translating soliton, possibly times a Euclidean factor.
        \end{enumerate}
    \end{enumerate}
\end{thm}

\begin{rem}[Models]
The ancient ovals in Theorem \ref{item:classification} are those constructed in \cite{White_NatureSingularities,HH_AncientSolutions,DH_AncientOvals}, while the flying wings are the translating solitons of Hoffman, Ilmanen, Mart\'{\i}n and White \cite{HIMW_FlyingWings1} (alternatively constructed in \cite[Proposition $6.6$]{BamlerLai_Classification}).
\end{rem}

\begin{rem}[Dimensions]\label{rem:classification_dimension}
Flying wings, which require $k\geq 2$ by \cite{BamlerLai_Classification}, can only occur when $n\geq 9$, as $4k<n$. When $2\leq n\leq 4$, the condition $4k<n$ forces $k=0$ and the flow is either a static plane or a family of 
    shrinking spheres. For $5\leq n\leq 8$ it forces $k\leq 1$, and the only possible models are the static plane, the sphere, the cylinder $\mathbb R\times S^{n-1}$, the ancient oval and the bowl soliton.
\end{rem}

    For the codimension reduction, Theorem \ref{item:codim_reduction} holds up to the Andrews--Baker threshold $\frac{4}{3n}$ in every dimension, improving the constant $c_n^{\mathrm{Naff}}$ of \cite[Corollary $1.2$]{LNZ_Planarity} when $2\leq n\leq 6$. For $n\in\{3,4\}$, Lee, Naff and Zhu already reach $\frac{4}{3n}$ in the proof of \cite[Corollary 1.5(3)]{LNZ_Planarity}, through their Proposition $4.1$ and 
\cite[Theorem $1.1$]{LynchNguyen_PinchedAncientSolutionsHighCodimMCF}, so the gain over their result concerns $n\in\{2,5,6\}$. Theorem \ref{item:embedding_noncollapsness} gives, without assuming that the flow arises as a blow-up limit, the conclusion of \cite[Corollary $1.6$]{BourniLangfordLynch_Noncollapsing} at $\frac{4}{3n}$ in every dimension. For uniformly $\frac{4}{3n}$-pinched flows, it also replaces the entropy and the curvature growth assumptions of the convexity result \cite[Theorem $1.3$]{LNZ_Planarity} with space-time properness. As for the classification, the constant  $\min\left\{c_n^{\mathrm{Naff}},\frac{1}{n-2}\right\}$ of \cite[Corollary $1.5(1)$]{LNZ_Planarity} equals $\frac{4}{3n}$ when $n\in\{7,8\}$, while it is 
smaller for $n\in\{5,6\}$ and for $n\geq 9$. For $5\leq n\leq 8$, Theorem \ref{item:classification} yields the same classification as the one in Corollary $1.5(1)$ of \cite{LNZ_Planarity}, while, as mentioned in Remark \ref{rem:classification_dimension}, for $n\geq 9$ further models appear.

\subsection{Interior planarity estimate}

Naff proved \cite{Naff_Planarity} that if a closed initial submanifold $M_0$ satisfies $|B|^2<c_n^{\mathrm{Naff}}|\mathbf{H}|^2$, then there exist constants $\sigma>0$ and $C<\infty$, both depending only on $M_0$, such that 
\begin{equation}\label{eq:planarity_estimate_original}
    |B_I|^2\leq C|\mathbf H|^{2-2\sigma}
\end{equation}
holds at every point of $M_t$, where $B_I$ is the component of the second fundamental form orthogonal to $\mathbf H$. For $2\leq n\leq 4$ one has $\frac{4}{3n}\leq\frac{1}{n-1}$, so closed flows satisfy the stronger umbilic estimate $|\mathring B|^2\leq C|\mathbf H|^{2-2\sigma}$; see \cite[Theorem $4$]{AndrewsBaker_MCFPinchedSubmanifolds}.

Very recently, Lynch \cite{Lynch_PlanarityEstimate} closed the gap regarding the constant $c_n^{\mathrm{Naff}}$ in dimensions $n\in\{5,6\}$, proving that the
planarity estimate \eqref{eq:planarity_estimate_original} holds up to $\frac{4}{3n}$ for closed flows in every dimension $n\geq 5$, for every $\sigma\in\left(0,\frac{1}{30}\right]$. The new key ingredient is a sharp Kato-type inequality that relates the twisting of the principal normal $\nu\doteq \mathbf H/|\mathbf H|$ to the gradient of the second fundamental form.

Independently of Lynch's work, the gap between $c_n^{\mathrm{Naff}}$ and $\frac{4}{3n}$ in dimensions $5$ and $6$ was shown to be a consequence of technical aspects of the approach in \cite{Naff_Planarity,LNZ_Planarity} both for self-shrinkers \cite{IRR_shrinkers} and for translators \cite{IRR_translators}, by purely elliptic approaches.

Under a pinching assumption, $B=0$ wherever $\mathbf H=0$, and there we define $B_I\doteq0$. Given $x_0\in\mathbb R^{n+p}$ and $R>0$ we set
\begin{equation}\label{eq:def_K_lambda}
    K_\lambda\doteq \left\{(q,t)\in M\times[-\lambda^2R^2,0]\:\middle |\:|x(q,t)-x_0|\leq\lambda R\right\}.
\end{equation}
If the flow is space-time proper, then every $K_\lambda$ is compact.

\begin{thm}[Interior planarity estimate]\label{thm:main_result_planarity_estimate}
    Let $n\geq 5$, $p\geq 1$, $\varepsilon_0\in\left(0,\frac{1}{3n}\right)$ and $\sigma\in \left(0,\frac{1}{6}\right]$. There exists a constant $C_1=C_1(n,\varepsilon_0,\sigma)>0$ with the following property. Let $R>0$, $x_0\in\mathbb R^{n+p}$ and $x:M^n\times[-R^2,0]\to\mathbb R^{n+p}$ be a MCF of a connected manifold $M$ without boundary. Suppose $K_1$ is compact and $|B|^2\leq \left(\frac{4}{3n}-\varepsilon_0\right)|\mathbf H|^2$ on $K_1$. Then
    \begin{equation*}
        |B_I|^2\leq C_1R^{-2\sigma}|\mathbf H|^{2-2\sigma}\ \text{ on }K_{1/2}. 
    \end{equation*}
\end{thm}

Choosing $R^2=t$ in Theorem \ref{thm:main_result_planarity_estimate} gives a parabolically scale-invariant planarity estimate for flows that start at $t=0$, hence recovering an analogue of \cite[Theorem $1.1$]{LNZ_Planarity} at $\frac{4}{3n}$ with bounded curvature replaced by space-time properness. Their exponent depends on $n$ and the pinching, while in our case every $\sigma\in\left(0,\frac{1}{6}\right]$ is admissible for every $\varepsilon_0$, and only $C_1$ depends on $\varepsilon_0$. Moreover, for closed flows we extend \cite[Corollary $1.2$]{Lynch_PlanarityEstimate} from $\sigma\in\left(0,\frac{1}{30}\right]$ to $\sigma\in\left(0,\frac{1}{6}\right]$.

\begin{cor}\label{cor:lynch_planarity_results_sigma_extension}
    Let $n\geq 5$, $p\geq 1$, $\varepsilon_0\in\left(0,\frac{1}{3n}\right)$ and $\sigma\in\left(0,\frac{1}{6}\right]$. Let $C_1=C_1(n,\varepsilon_0,\sigma)$ be the constant of Theorem \ref{thm:main_result_planarity_estimate}. 
    \begin{enumerate}
        \item Let $x:M^n\times[0,T)\to\mathbb R^{n+p}$ be a space-time proper mean curvature flow of a connected manifold without boundary. Assume $|B|^2\leq(\frac4{3n}-\varepsilon_0)|\mathbf H|^2$ on $M\times[0,T)$. Then
        \begin{equation*}
            |B_I|^2\leq C_1\,t^{-\sigma}|\mathbf H|^{2-2\sigma}\ \text{ on }M\times(0,T).
        \end{equation*}
        \item If $M$ is closed, it is enough to assume $|\mathbf H|>0$ and $|B|^2\leq(\frac4{3n}-\varepsilon_0)|\mathbf H|^2$ on $M\times\{0\}$. Moreover, in this case,
        \begin{equation*}
            |B_I|^2\leq C_2|\mathbf H|^{2-2\sigma}\ \text{ on }M\times[0,T),
        \end{equation*}
        where $C_2=C_2(n,\varepsilon_0,\sigma,M_0)$.
    \end{enumerate}
\end{cor}

\subsection{Proofs of the main results}

\subsubsection*{Planarity estimate} 

The
evolution equations of $|B_I|^2$ and $f\doteq \frac{4}{3n}|\mathbf H|^2-|B|^2$ give, when combined,
\begin{equation*}
    (\partial_t-\Delta)|B_I|^2-(1-\sigma)u(\partial_t-\Delta)f=-2(\text{reaction terms})-2(\text{gradient terms}),
\end{equation*}
where $u\doteq |B_I|^2/f$. Under Naff’s gradient estimates, the gradient terms can be negative. Following \cite{IRR_translators}, we add to them the gradient terms produced by the logarithmic derivative of $\omega_\sigma\doteq fu^{\frac{1}{\sigma}}$, and we bound the sum below by $-\widetilde C\sigma^2|B_I|^2|\nabla\log\omega_\sigma|^2$. Now suppose that a function $\omega\geq 0$, smooth on $K_1\cap\{\omega>0\}$, satisfies a differential inequality of the form
\begin{equation}\label{eq:needed_diff_inequality_proof_sketch}
    (\partial_t-\Delta)(\log\omega)\leq -\beta\omega+\gamma|\nabla\log\omega|^2,
\end{equation}
with $\beta>0$ and $\gamma\geq 0$. Then $\omega\leq C(n,\beta,\gamma)R^{-2}$ on $K_{1/2}$. The proof of this fact, which is the content of Lemma \ref{lem:localization_lemma}, is a localization based on the cutoff argument of Ecker--Huisken \cite{EckerHuisken_InteriorEstimates}. Hence, the gradient error term is harmless, because at the maximum of the localized function, $\nabla\log\omega$ equals minus the gradient of the logarithm of the cutoff, which is controlled. No entropy or curvature bounds are needed, since the only used ingredients are the identity $(\partial_t-\Delta)|x-x_0|^2=-2n$ and the compactness of the parabolic cylinders. 

\subsubsection*{Codimension reduction}

Let $n\geq 5$. Unless the flow is a static plane, the strong maximum principle applied to $f$ yields $|\mathbf H|>0$ everywhere. Letting $R\to\infty$ in the planarity estimate one gets $B_I\equiv 0$, and hence each $M_t$ is contained in an affine $(n+1)$-dimensional subspace $A_t$ by \cite[Proposition $2.5$]{Naff_Planarity}. To show the time-independence of the affine subspaces $A_t$, one proves that containment in an affine subspace propagates forward in time along space-time proper flows. For $2\leq n\leq 4$, the compactness criterion \cite[Proposition $8.1$]{IRR_translators} (see also \cite[Proposition $4.1$]{LNZ_Planarity}) reduces the problem to compact solutions, which are shrinking spheres by \cite[Theorem $1.1$]{LynchNguyen_PinchedAncientSolutionsHighCodimMCF}.

\subsubsection*{Classification}

Assume the flow is a family of hypersurfaces with $H\doteq |\mathbf H|>0$ contained in a fixed, affine subspace $\mathcal A\cong\mathbb R^{n+1}$. 
The pair of derivative estimates $|\nabla h|\leq \Gamma_1 H^2$ and $|\nabla^2 h|\leq\Gamma_2H^3$, obtained by localizing the gradient estimate of \cite[Theorem $4.1$]{LynchNguyen_HighCodimMCFSurgery} and its second order analogue, gives, as in \cite[Proposition $2.6$]{Naff_SingularityModels}, smooth limits of rescalings at points where $\kappa_1/H$ approaches its infimum (where $\kappa_1$ is the smallest principal curvature). Then, White's maximum principle \cite{White_NatureSingularities} shows that the infimum is nonnegative. As a result, the time slices bound convex bodies and by \cite{BourniLangfordTinaglia_ConvexSolutions} the blow-down is a sphere, a cylinder or a hyperplane of multiplicity two. This last case is excluded by the pinching together with \cite{BourniLangfordLynch_Noncollapsing}, and the classification follows from \cite[Theorem $1.2$]{BamlerLai_Classification}. 

\subsection{Paper structure}

Section \ref{sec:notation_and_equations_sec_2} fixes the notation and recalls the evolution equations used through the rest of the paper. Section \ref{sec:localization_lemma_sec_3} proves the localization lemma, Section \ref{sec:pointwise_diff_inequality_sec_4} proves \eqref{eq:needed_diff_inequality_proof_sketch} for $\omega_{1/6}$, and Section \ref{sec:result_proofs} proves Theorem \ref{thm:main_result_planarity_estimate} and Corollary \ref{cor:lynch_planarity_results_sigma_extension}, with Remark \ref{rem:lynch_approach} on Lynch's approach. Section \ref{sec:ancient_sec_6} proves Theorem \ref{item:codim_reduction}, and, lastly, Section \ref{sec:classification_sec_7} proves Theorems \ref{item:embedding_noncollapsness} and \ref{item:classification}.

\subsection*{Acknowledgements}

The author is deeply grateful to Stephen Lynch for his generous help and suggestions, in particular for suggesting the classification result of Section \ref{sec:classification_sec_7}. The author also thanks Debora Impera, Stephen Lynch and Michele Rimoldi for encouraging the publication of this work.

\subsection*{Added note}

The planarity estimate and the codimension reduction were almost completed when Cheng \cite{Cheng_PlanarityEstimate} posted a preprint containing, among other results, a local planarity estimate in arbitrary codimension without curvature assumptions, and the resulting codimension reduction for space-time proper ancient solutions. These parts were carried out independently.

Both works localize the quantity $\omega_\sigma$ of \cite{LNZ_Planarity} by a cutoff argument at a maximum point. Cheng assumes uniform $c_n^{\mathrm{Naff}}$-pinching, under which the gradient terms are nonnegative by \cite{Naff_Planarity,LNZ_Planarity}. His argument then has the same structure as Remark \ref{rem:lynch_approach}, and his exponent depends on $n$ and the pinching. We also rely only on Naff’s estimates, but we do not need the gradient terms to be nonnegative. This allows us to reach $\frac{4}{3n}$ for every $n\geq5$ with a range of exponents independent of $\varepsilon_0$. As a result, we improve the pinching constant of the local estimate for $n\in\{5,6\}$ and that of the codimension reduction for $2\leq n\leq 6$. 

\subsection*{AI Disclosure} 

Claude Opus $5.5$ was used in preparing the manuscript. In the proof of Lemma \ref{lem:forward_confinement}, it proposed the exponential weight $e^{-(t-t_0)}$ in the definition of $\rho$, as already appears in the proof of \cite[Theorem $1.3$]{Ma_ParabolicOmoriYau}. The weight produces the term $-\rho$ in \eqref{eq:ineq_rho_forward_lemma}, which, combined with \cite[Theorem $1.2$]{Ma_ParabolicOmoriYau}, gives the contradiction. Furthermore, it suggested the right denominator $\Psi$ needed in the proof of the second estimate in Proposition \ref{prop:derivative_estimates_ancient}. Claude also checked the mathematical arguments and improved the language and exposition. The author independently verified all mathematical content and takes full responsibility for it.

\section{Notation and known results}\label{sec:notation_and_equations_sec_2}

We use the same notation adopted in \cite{IRR_shrinkers,IRR_translators}, which we quickly recall. Let $M^n$ be a connected $n$-dimensional manifold immersed in the Euclidean space $\mathbb R^{n+p}$. The second fundamental form $B$ is defined as the normal component of the flat ambient covariant derivative $\overline \nabla$, and $\mathbf H\doteq \operatorname{tr}_gB$ is the mean curvature vector, where $g$ is the induced metric on $M$. Where needed, let $\{\nu_\alpha\}_{\alpha=n+1}^{n+p}$ be a local orthonormal frame for the normal bundle. Whenever $|\mathbf H|>0$, we take $\nu_{n+1}=\nu\doteq\mathbf H/|\mathbf H|$. For each $\alpha$, set $h^\alpha\doteq \langle B,\nu_\alpha\rangle$, so that the second fundamental form orthogonally decomposes as $B=B_\nu+B_I$, where $B_\nu\doteq h^{n+1}\otimes\nu$ and $B_I\doteq \sum_{\alpha\geq n+2}h^\alpha\otimes\nu_\alpha$. The second term, $B_I$, is traceless. The tracefree principal part is defined as $\mathring B_\nu\doteq \mathring h^{n+1}\otimes\nu$, with $\mathring h^{n+1}\doteq h^{n+1}-\frac{|\mathbf H|}{n}g$. Since $\operatorname{tr}h^{n+1}=|\mathbf H|$, the tensor $\mathring h^{n+1}$ is indeed tracefree. 

We denote by $\nabla$ and $\nabla^\perp$ the connections induced on $TM$ and $NM$, 
namely $\nabla_XY\doteq(\overline\nabla_XY)^\top$ and $\nabla^\perp_X\xi\doteq(\overline\nabla_X\xi)^\perp$ for $X,Y\in\Gamma(TM)$ and $\xi\in\Gamma(NM)$, and we still write $\nabla^\perp$ for the connection they induce on $(T^*M)^{\otimes k}\otimes NM$.

Throughout this paper we assume that the flow satisfies the uniform pinching with the Andrews--Baker constant $\frac{4}{3n}$, i.e.
\begin{equation}\label{eq:pinching}
    |B|^2\leq\left(\frac{4}{3n}-\varepsilon_0\right)|\mathbf H|^2,
\end{equation}
for some $\varepsilon_0>0$. Thus, it is natural to define the pinching function $f$ as $f\doteq\frac{4}{3n}|\mathbf H|^2-|B|^2$, which, along a mean curvature flow, satisfies
\begin{equation}\label{eq:evol_equation_f}
    (\partial_t-\Delta)f=2(\mathcal R_f+\mathcal G_f),
\end{equation}
as derived in \cite[$(4.2)$]{Naff_Planarity} and reproved in \cite[Corollary 5.7]{IRR_translators} in the elliptic setting of translators. Here, $\Delta=\Delta_{g(t)}$ is the Laplacian of the evolving metric. The two terms are
\begin{align*}
    \mathcal R_f&\doteq \frac{4}{3n}|S_{\mathbf H}|^2-\sum_{\alpha,\beta}[\operatorname{tr}(h^\alpha h^\beta)]^2-\sum_{\alpha,\beta}|[h^\alpha,h^\beta]|^2,
    \\\mathcal G_f&\doteq |\nabla^\perp B|^2-\frac{4}{3n}|\nabla^\perp\mathbf H|^2,
\end{align*}
where $S_{\mathbf H}$ is the shape operator in the normal direction $\mathbf H$. It is defined by the relation $\langle S_\xi X,Y\rangle=\langle B(X,Y),\xi\rangle$, for $\xi\in\Gamma(NM)$.
Since $|B|^2=|\mathring B_\nu|^2+\frac{1}{n}|\mathbf H|^2+|B_I|^2$,
\begin{equation}\label{eq:Q_A_I_trlessh_identity}
    f+|B_I|^2+|\mathring B_\nu|^2=\frac{1}{3n}|\mathbf H|^2.
\end{equation}
As a consequence, if $|\mathbf H|>0$ and $|B|^2\leq\left(\frac{4}{3n}-\varepsilon_0\right)|\mathbf H|^2$ for some $\varepsilon_0\in\left(0,\frac{1}{3n}\right)$, then
\begin{equation}\label{eq:bound_Q_A_I_u}
    0<\varepsilon_0|\mathbf H|^2\leq f\leq \frac{|\mathbf H|^2}{3n},\qquad |B_I|^2\leq\frac{|\mathbf H|^2}{3n}-f\leq\frac{|\mathbf H|^2}{3n},\qquad u\leq \overline u\doteq \frac{1-3n\varepsilon_0}{3n\varepsilon_0},
\end{equation}
where
\begin{equation}\label{eq:def_u}
    u\doteq \frac{|B_I|^2}{f}. 
\end{equation}
With the convention $B_I\doteq0$ where $\mathbf H=0$, the bound $|B_I|^2\leq\frac{1}{3n}|\mathbf H|^2$ holds everywhere, and $|B_I|^2$ is continuous.

The evolution equation of $f$ satisfies the following inequality.
\begin{lem}\label{lem:inequality_evol_equation_f}
    Let $x$ be a mean curvature flow, and assume that $|B|^2\leq\frac{4}{3n}|\mathbf H|^2$. Then, with $\delta_n\doteq\frac{5n-8}{9n}$,
    \begin{equation}\label{eq:inequality_evol_f}
        (\partial_t-\Delta)f\geq 2\delta_n|\nabla^\perp B|^2+6f^2\geq 0.
    \end{equation}
\end{lem}
\begin{proof}
    Since $\nabla^\perp B$ is, as a consequence of Codazzi equations, totally symmetric, it satisfies the well-known Kato inequality $|\nabla^\perp B|^2\geq \frac{3}{n+2}|\nabla^\perp\mathbf H|^2$; see, e.g., \cite{Huisken_FlowConvexSurfacesIntoSpheres} and \cite[Lemma 2.2]{Lynch_PlanarityEstimate}. So, $\mathcal G_f\geq\delta_n|\nabla^\perp B|^2\geq 0$. Where $|\mathbf H|>0$, $\mathcal R_f\geq 6f|B_I|^2+4f|\mathring B_\nu|^2+3f^2\geq 3f^2$ as proved in \cite[$(4.13)$]{Naff_Planarity}. On the other hand, where $|\mathbf H|=0$ the pinching forces $B=0$, and hence $\mathcal R_f=0=3f^2$. The claim follows from \eqref{eq:evol_equation_f}.
\end{proof}

As shown in \cite[Proposition $3.2$]{Naff_Planarity} and rederived in \cite[Proposition 5.23]{IRR_translators} for translators, where $|\mathbf H|>0$ the evolution equation of $|B_I|^2$ is 
\begin{equation}\label{eq:evol_equation_squared_norm_BI}
    (\partial_t-\Delta)|B_I|^2=-2|\nabla^\perp B_I|^2+2\mathcal R_I+4\sum_{i,j,k}Q_{ijk}C_{ijk},
\end{equation}
with
\begin{align*}
    \mathcal R_I&\doteq \sum_{\alpha,\beta\geq n+2}[\operatorname{tr}(h^\alpha h^\beta)]^2+\sum_{\alpha,\beta\geq n+2}|[h^\alpha,h^\beta]|^2+\sum_{\beta\geq n+2}|[h^{n+1},h^\beta]|^2,
    \\Q_{ijk}&\doteq\nabla_kh^{n+1}_{ij}-h^{n+1}_{ij}\nabla_k\log|\mathbf H|,
    \\ C_{ijk}&\doteq\bigl\langle B_I(e_i,e_j),\nabla^\perp_k\nu\bigr\rangle.
\end{align*}
Combining \eqref{eq:evol_equation_f} and \eqref{eq:evol_equation_squared_norm_BI},
\begin{equation}\label{eq:key_identity}
    \begin{aligned}
        (\partial_t-\Delta)|B_I|^2-(1-\sigma) u(\partial_t-\Delta)f&=-2[(1-\sigma) u\mathcal R_f-\mathcal R_I]
        \\&\qquad-2\Big[|\nabla^\perp B_I|^2+(1-\sigma) u\mathcal G_f-2\sum Q_{ijk}C_{ijk}\Big],
    \end{aligned}
\end{equation}
for $\sigma\in(0,1)$,
which is the parabolic counterpart to \cite[(7.1)]{IRR_translators} with $\theta=1-\sigma$. The two brackets collect, respectively, the reaction and the gradient terms. The planarity estimates of \cite{Naff_Planarity,LNZ_Planarity,Lynch_PlanarityEstimate,IRR_translators} all rest on lower bounds for these two brackets. The difference is in the choice of $\sigma$ and in the treatment of the gradient bracket.  

In \cite{Naff_Planarity}, both brackets are shown to be pointwise nonnegative. For the pinching function $f$, this works only when $n\geq 8$, while in lower dimensions this gives rise to the constant $c_n^{\mathrm{Naff}}$, as explained in detail in \cite[Appendix B2]{LNZ_Planarity} and in \cite[Subsection $7.1$]{IRR_translators}. 

In \cite{Lynch_PlanarityEstimate}, the author takes $\sigma\in\left(0,\frac{1}{30}\right]$ and, for $n\geq5$, makes the gradient bracket nonnegative through a sharp Kato-type inequality; see Remark \ref{rem:lynch_approach}.

In \cite{IRR_translators}, one takes $\sigma=\frac{1}{6}$. The reaction bracket is bounded below by $\frac{7}{2}|B_I|^4$, see \cite[Lemma 7.2]{IRR_translators}, while the gradient bracket is not shown to be nonnegative. Instead, the gradient terms produced by the logarithmic derivative of $\omega_{1/6}\doteq |B_I|^{12}/f^5$ are added to it, and the sum is bounded below by a negative multiple of $|B_I|^2|\nabla\log\omega_{1/6}|^2$; see \cite[Lemma $7.11$]{IRR_translators}. This lower bound vanishes at critical points of $\omega_{1/6}$. However, it is proved only where $u$ is bounded away from zero. In Subsection \ref{subsec:removing_boundness} we remove this restriction and we adapt the estimate to the parabolic setting.

\section{Localization lemma}\label{sec:localization_lemma_sec_3}

Fix a point $x_0\in\mathbb R^{n+p}$, and consider the function $r(x)\doteq |x-x_0|^2$. Along a mean curvature flow $x$, it satisfies
\begin{equation}\label{eq:evolution_equations_r}
    (\partial_t-\Delta)r=-2n,\qquad |\nabla r|^2=4|(x-x_0)^\top|^2\leq 4r.
\end{equation}

\begin{lem}\label{lem:localization_lemma}
    Let $R>0$, $x_0\in\mathbb R^{n+p}$, and let $x:M^n\times[-R^2,0]\to\mathbb R^{n+p}$ be a MCF with $K_1$ compact. Let $\omega\geq 0$ be continuous on $K_1$ and smooth on $K_1\cap\{\omega>0\}$, and assume that on the latter set it satisfies
    \begin{equation}\label{eq:hypothesis_bound_evol_equation_omega}
        (\partial_t-\Delta)(\log\omega)\leq -\beta\omega+\gamma|\nabla\log\omega|^2,\qquad \beta>0,\:\gamma\geq 0.
    \end{equation}
    Then
    \begin{equation*}
        \omega\leq \frac{64(4n+9+16\gamma)}{27\beta\;R^2}\ \text{ on }K_{1/2},
    \end{equation*}
    where the sets $K_\lambda$ are defined as in \eqref{eq:def_K_lambda}.
\end{lem}

\begin{proof}
    We follow the cutoff argument presented in the proof of \cite[Theorems $2.1$ and $3.1$]{EckerHuisken_InteriorEstimates}. 
    Set $\Xi\doteq (t+R^2)\eta\omega\geq 0$, where $\eta=\eta(r)=(R^2-r)^2$. The function $\Xi$ is continuous on $K_1$ and vanishes on the parabolic boundary $\partial_PK_1=\{(q,t)\in K_1\:|\:t=-R^2\text{ or }r(q,t)=R^2\}$. If $\Xi\equiv 0$ on $K_1$ there is nothing to prove. Otherwise, by compactness of $K_1$, $\Xi$ attains a positive maximum on $K_1$ at some point $(q_*,t_*)$ in the interior, i.e. with $t_*>-R^2$ and $r(q_*,t_*)<R^2$. At that point, $\nabla\log\Xi=0$ and $(\partial_t-\Delta)(\log\Xi)\geq 0$. Moreover, by \eqref{eq:evolution_equations_r},
    \begin{equation}\label{eq:evol_equations_log_eta}
        (\partial_t-\Delta)(\log\eta)=\frac{4n(R^2-r)+2|\nabla r|^2}{\eta},\qquad |\nabla\log\eta|^2=\frac{4|\nabla r|^2}{\eta}. 
    \end{equation}
    Keeping in mind that $\nabla\log\Xi=0$ and $\nabla\log\omega=-\nabla\log\eta$ at $(q_*,t_*)$, there it holds that
    \begin{equation*}
        0\leq(\partial_t-\Delta)(\log\Xi)=\frac{1}{t+R^2}+(\partial_t-\Delta)(\log\eta)+(\partial_t-\Delta)(\log\omega).
    \end{equation*}
    Substituting \eqref{eq:hypothesis_bound_evol_equation_omega} and \eqref{eq:evol_equations_log_eta},
    \begin{equation*}
        0\leq \frac{1}{t+R^2}+\frac{4n(R^2-r)+(2+4\gamma)|\nabla r|^2}{\eta}-\beta\omega\ \text{ at }(q_*,t_*).
    \end{equation*}
    Multiplying by $(t+R^2)\eta$ yields
    \begin{equation*}
        \beta\Xi\leq \eta+4n(R^2-r)(t+R^2)+(2+4\gamma)(t+R^2)|\nabla r|^2\ \text{ at }(q_*,t_*).
    \end{equation*}
    Since $t+R^2\leq R^2$, $\eta\leq R^4$ and $|\nabla r|^2\leq 4r\leq 4R^2$,
    \begin{equation*}
        \beta\Xi\leq (4n+9+16\gamma)R^4,
    \end{equation*}
    which holds on the whole $K_1$ since $(q_*,t_*)$ was the maximum point.
    Over the smaller set $K_{1/2}$, $(t+R^2)\eta\geq\frac{27}{64}R^6$, which gives the claim together with the previous inequality.
\end{proof}

\section{A differential inequality for the planarity quantity}\label{sec:pointwise_diff_inequality_sec_4}

Assume, in this section, that $n\geq 5$, $|\mathbf H|>0$, $|B_I|>0$ and $|B|^2\leq\left(\frac{4}{3n}-\varepsilon_0\right)|\mathbf H|^2$, with $\varepsilon_0\in\left(0,\frac{1}{3n}\right)$. Define, for $\sigma\in(0,1)$,
\begin{equation*}
    \omega_\sigma\doteq fu^\frac{1}{\sigma},\qquad \zeta\doteq\sigma\nabla\log\omega_\sigma.
\end{equation*}

\begin{lem}\label{lem:identity_log_omega}
    At every point where $|B_I|>0$, the following equality
    \begin{align*}
        \sigma(\partial_t-\Delta)(\log\omega_\sigma)&=\frac{1}{|B_I|^2}\left[(\partial_t-\Delta)|B_I|^2-(1-\sigma) u(\partial_t-\Delta)f\right]+|\zeta|^2
        \\&\qquad+2(1-\sigma)\langle\zeta,\nabla\log f\rangle-\sigma(1-\sigma)|\nabla\log f|^2
    \end{align*}
    holds. As a consequence, inserting \eqref{eq:key_identity}, 
        \begin{equation}\label{eq:translator_approach}
            \sigma(\partial_t-\Delta)(\log\omega_\sigma)=-\frac{2}{|B_I|^2}[(1-\sigma) u\mathcal R_f-\mathcal R_I]-\frac{2}{|B_I|^2}\mathcal G_\zeta,
        \end{equation}
        where $\mathcal G_\zeta$ is as in \cite[$(7.10)$]{IRR_translators}.
\end{lem}
\begin{proof}
    To derive the first identity, apply the chain rule $(\partial_t-\Delta)(\log\varphi)=\frac{1}{\varphi}(\partial_t-\Delta)\varphi+|\nabla\log\varphi|^2$ to $\sigma\log\omega_\sigma=\log|B_I|^2-(1-\sigma)\log f$, and use that $\nabla\log|B_I|^2=\zeta+(1-\sigma)\nabla\log f$.
\end{proof}

\subsection{An algebraic result}

We use the notation of \cite{IRR_translators} with $\theta\doteq 1-\sigma$. To avoid clashes with time $t$, the cutoff $\eta$ of Lemma \ref{lem:localization_lemma} and the pinching constant $c$, we write $\mathsf s,\mathsf t,\mathsf c$ and $\hat\eta$ for the quantities $s,t,c$ and $\eta$ of Lemma $7.11$ and Appendix A of \cite{IRR_translators}. For $n\geq 5$ and $\sigma\in(0,1)$, set
\begin{equation*}
\chi \doteq (1-\sigma) \frac{5n-8}{n+2}, \quad
\tau \doteq \frac{3(n-1)}{n+2}, \quad
\hat\eta \doteq \frac{2\tau}{\chi}, \quad
\mu \doteq \frac{1}{2\sigma(1-\sigma)}.
\end{equation*}
Conditions (C$1$) and (C$2$) are those of \cite[Lemma A.$1$]{IRR_translators}, while in condition (C$3$) of the same paper we assume the strict inequality. Specifically,
\begin{equation*}
    \text{(C1)}\ \ 2\hat\eta+\tau\geq\mu+1,
    \qquad\text{(C2)}\ \ \tau+1\geq\frac{1-\sigma}{2\sigma},
    \qquad\text{(C3)}\ \ \chi>2,
\end{equation*}
so that, thanks to the last condition, $\tau-\hat\eta=\tau\frac{\chi-2}{\chi}>0$. For $\mathsf t\in[0,1)$ and $u>0$, let
\begin{align*}
    \Phi_0(\mathsf t)&\doteq\tau-\hat\eta+\hat\eta\,\mathsf t(2-\mathsf t)+\frac{(\tau-1)\mathsf t}{1-\mathsf t}-\mu\,\mathsf t(1-\mathsf t),
    \\\Phi(\mathsf t,u)&\doteq\Phi_0(\mathsf t)+u\left[\tau-1+2\mathsf t+\frac{(\tau-1)\mathsf t}{1-\mathsf t}+\frac{\chi(1+\mathsf tu)}{1-\mathsf t}\right]-\mu\,\mathsf t(1-\mathsf t)\left[\frac{(1+(1-\sigma)u)^2}{1+u}-1\right],
\end{align*}
as in \cite[(A.3)]{IRR_translators}.

Assuming conditions (C$1$)--(C$3$), then for every $\mathsf t\in[0,1)$ and $u>0$, the proof of \cite[Lemma A.$1$]{IRR_translators} shows that $\Phi_0(\mathsf t)\geq \tau-\hat\eta$ and
\begin{equation*}
    \Phi(\mathsf t,u)\geq\Phi_0(\mathsf t)+u\left[\tau+\chi-1+\mathsf t\left(1+\tau-\frac{1-\sigma}{2\sigma}(1-\mathsf t)\right)\right].
\end{equation*}

Since the coefficient $1+\tau-\frac{1-\sigma}{2\sigma}(1-\mathsf t)$ of $\mathsf t$ is nonnegative thanks to condition (C$2$) and the range of $1-\mathsf t$ (i.e. $1-\mathsf t\in(0,1]$), from the cited results it follows that
\begin{equation}\label{eq:key_algebraic_inequality}
    \Phi(\mathsf t,u)\geq \tau-\hat\eta+u(\tau+\chi-1). 
\end{equation}

\subsection{Dropping the assumption \texorpdfstring{$u\geq \underline u>0$}{u >= underline u > 0}}\label{subsec:removing_boundness}

At points where $|\mathbf H|>0$ and $|B_I|>0$, let
\begin{equation*}
    \mathsf s\doteq \frac{3nf}{|\mathbf H|^2},\qquad \mathsf t\doteq \frac{3n|\mathring B_\nu|^2}{|\mathbf H|^2}.
\end{equation*}

By \eqref{eq:Q_A_I_trlessh_identity} and \eqref{eq:def_u}, $\mathsf s(1+u)+\mathsf t=1$, from which, in particular, one deduces that $\mathsf t\in[0,1)$. Let 
\begin{equation*}
    \Lambda\doteq|B_I|^2\,\frac{3n|\mathring B_\nu|}{|\mathbf H|^2}\,\frac{1+u}{\sigma u}.
\end{equation*}
As shown in \cite[Lemma $7.11$]{IRR_translators},
\begin{equation}\label{eq:bound_G_zeta_Lemma_7.11}
    \mathcal G_\zeta\geq\left(\chi u+\frac{\Phi(\mathsf t,u)}{u}\right)|\mathsf c|^2+\Lambda\langle\zeta,\mathsf c\rangle-\frac{|B_I|^2}{2\sigma}|\zeta|^2,
\end{equation}
where $\mathsf c$ is a suitable vector field defined therein. Indeed, in the proof of \cite[Lemma $7.11$]{IRR_translators}, the hypothesis $u\geq \underline u>0$ is used only to ensure that $|B_I|>0$ and in a final step. Before that passage the proof bounds $\mathcal G_\zeta\geq\mathcal G_\zeta^\perp+\mathcal G_\zeta^\parallel$, with $\mathcal G_\zeta^\perp\geq 0$ by (C$3$) and $\mathcal G_\zeta^\parallel$ greater than or equal to the right-hand side of \eqref{eq:bound_G_zeta_Lemma_7.11}. The coefficient of $|\mathsf c|^2$ there equals $\chi u+\frac{\Phi(\mathsf t,u)}{u}$ thanks to the substitution $\mathsf s=\frac{1-\mathsf t}{1+u}$. All the estimates involved are pointwise and use neither the translator equation nor the value $\sigma=\frac{1}{6}$. In the final step, \cite{IRR_translators} bounds the coefficient of $|\mathsf c|^2$ from below by $\chi u$, using that $\Phi(\mathsf t,u)\geq 0$, which makes $u\geq\underline u>0$ needed to control $\frac{\Lambda^2}{4\chi u}$. We use the better bound \eqref{eq:key_algebraic_inequality} instead.

\begin{lem}\label{lem:independence_from_underline_u}
    Assume (C$1$)--(C$3$) hold. At every point where $|\mathbf H|>0$ and $|B_I|>0$,
    \begin{equation*}
        \mathcal G_\zeta\geq-\widetilde C|B_I|^2|\zeta|^2,\qquad \widetilde C\doteq \frac{1}{16\sigma^2(\tau-\hat\eta)}+\frac{1}{2\sigma}.
    \end{equation*}
\end{lem}

\begin{proof}
    Thanks to \eqref{eq:key_algebraic_inequality}, the coefficient of $|\mathsf c|^2$ in \eqref{eq:bound_G_zeta_Lemma_7.11} is at least
    \begin{equation*}
        d\doteq \chi u+\tau+\chi-1+\frac{\tau-\hat\eta}{u}>0.
    \end{equation*}
    Completing the square as in \cite[Lemma $7.11$]{IRR_translators} with $d$ in place of $\chi u$,
    \begin{equation*}
        d|\mathsf c|^2+\Lambda\langle\zeta,\mathsf c\rangle=\left|\sqrt d\mathsf c+\frac{\Lambda}{2\sqrt d}\zeta\right|^2-\frac{\Lambda^2}{4d}|\zeta|^2\geq -\frac{\Lambda^2}{4d}|\zeta|^2.
    \end{equation*}
    Notice that $|B_I|^2=uf=\frac{\mathsf su}{3n}|\mathbf H|^2$, which, together with $\mathsf s(1+u)=1-\mathsf t$, implies
    \begin{equation*}
        \Lambda^2=|B_I|^2\frac{\mathsf t(1-\mathsf t)(1+u)}{\sigma^2u}\leq |B_I|^2\frac{1+u}{4\sigma^2 u},
    \end{equation*}
    since $\mathsf t(1-\mathsf t)\leq\frac{1}{4}$ for every $\mathsf t\in[0,1)$. Furthermore,
    \begin{equation*}
        \frac{1+u}{du}=\frac{1+u}{\chi u^2+(\tau+\chi-1)u+\tau-\hat\eta}\leq\frac{1}{\tau-\hat\eta}. 
    \end{equation*}
    The last inequality is equivalent to $u(\chi u+\chi-1+\hat \eta)\geq 0$, which holds true since $u>0$, $\chi>2$ and $\hat\eta>0$. Thus,
    \begin{equation*}
        \mathcal G_\zeta\geq -\left(\frac{\Lambda^2}{4d}+\frac{|B_I|^2}{2\sigma}\right)|\zeta|^2\geq -\widetilde C|B_I|^2|\zeta|^2.
    \end{equation*}
\end{proof}

\subsection{The differential inequality for \texorpdfstring{$\sigma=\frac{1}{6}$}{sigma=1/6}}

\begin{lem}\label{lem:result_via_Naff_inequalities}
    Let $\sigma=\frac{1}{6}$, with $\tau$ and $\hat \eta$ the corresponding constants. At every point where $|\mathbf H|>0$ and $|B_I|>0$, inequality \eqref{eq:hypothesis_bound_evol_equation_omega} holds for $\omega_{1/6}$ with 
    \begin{equation*}
        \beta=42\overline u^{\,-5},\qquad \gamma=1+\frac{3}{4(\tau-\hat\eta)}. 
    \end{equation*}
\end{lem}

\begin{proof}
    For $\sigma=\frac{1}{6}$ and $n\geq 5$ conditions (C$1$)--(C$3$), the last of which is equivalent to $13n>64$, hold, as shown in \cite[Remark A.$4$]{IRR_translators}. So, Lemma \ref{lem:independence_from_underline_u} applies with $\widetilde C=\frac{9}{4(\tau-\hat\eta)}+3$. Inserting this bound and $(1-\sigma)u\mathcal R_f-\mathcal R_I\geq\frac{7}{2}|B_I|^4$, which is a consequence of \cite[Lemma $7.2$]{IRR_translators}, into \eqref{eq:translator_approach},
    \begin{equation*}
        \frac{1}{6}(\partial_t-\Delta)(\log\omega_{1/6})\leq -7|B_I|^2+2\widetilde C|\zeta|^2.
    \end{equation*}
    Recalling that $\zeta=\frac{1}{6}\nabla\log\omega_{1/6}$, this gives
    \begin{equation*}
        (\partial_t-\Delta)(\log\omega_{1/6})\leq -42|B_I|^2+\frac{\widetilde C}{3}|\nabla\log \omega_{1/6}|^2. 
    \end{equation*}
    Lastly, $|B_I|^2=u^{-5}\omega_{1/6}\geq \overline u^{\,-5}\omega_{1/6}$ by \eqref{eq:bound_Q_A_I_u}, and $\frac{\widetilde C}{3}=\gamma$. 
\end{proof}

\begin{rem}\label{rem:other_values_sigma_potential}
    The value $\sigma=\frac{1}{6}$ can actually be improved, as the proof of Lemma \ref{lem:result_via_Naff_inequalities} uses $\sigma=\frac{1}{6}$ only through (C$1$)--(C$3$) and the reaction estimate. For $\sigma\in(0,1)$, \cite[$(4.13)$]{Naff_Planarity} and \cite[$(6.3)$]{IRR_translators} give
    \begin{equation*}
        (1-\sigma)u\mathcal R_f-\mathcal R_I\ge\Big(\tfrac92-6\sigma\Big)|B_I|^4+2(1-2\sigma)|\mathring B_\nu|^2|B_I|^2+3(1-\sigma)f|B_I|^2.
    \end{equation*}
    So, if $\sigma\leq\frac{1}{2}$ and (C$1$)--(C$3$) hold, Lemma \ref{lem:result_via_Naff_inequalities} holds for $\omega_\sigma$ with the same proof, with $\beta$ and $\gamma$ depending on $\sigma$.
    Conditions (C$1$)--(C$3$) read as
    \begin{align*}
        \text{(C$1$)}&\Longleftrightarrow \frac{24(n-1)}{5n-8}\sigma+\frac{2(2n-5)}{n+2}\sigma(1-\sigma)\geq 1,
        \\\text{(C$2$)}&\Longleftrightarrow \sigma\geq \frac{n+2}{9n},
        \\\text{(C$3$)}&\Longleftrightarrow \sigma<\frac{3n-12}{5n-8}.
    \end{align*}
    The left-hand side in (C$1$) is increasing in $\sigma\in\left[0,\frac{1}{2}\right]$, and the inequality holds at $\sigma=\frac{1}{6}$, and hence for $\sigma\in\left[\frac{1}{6},\frac{1}{2}\right]$, by \cite[Remark A.$4$]{IRR_translators}. Moreover, since $\frac{3n-12}{5n-8}>\frac{1}{6}>\frac{n+2}{9n}$ for all $n\geq 5$, the three conditions (C$1$)--(C$3$), together with $\sigma\leq\frac{1}{2}$, are satisfied for every $\sigma\in\left[\frac{1}{6},\sigma_n\right)$, where
    \begin{equation*}
        \sigma_n\doteq\min\left\{\frac{3n-12}{5n-8},\frac{1}{2}\right\}.
    \end{equation*}
    Then, interpolating as in the proof below, Theorem \ref{thm:main_result_planarity_estimate} holds for $\sigma\in(0,\sigma_n)$. Notice that $\sigma_n=\frac{1}{2}$ for $n\geq 16$, while for $n=5$ the gain over $\frac{1}{6}$ is minimal, as $\sigma_5=\frac{3}{17}$. 
\end{rem}

\section{Proof of the interior planarity estimate}\label{sec:result_proofs}

\begin{proof}[Proof of Theorem \ref{thm:main_result_planarity_estimate}]
    By \eqref{eq:bound_Q_A_I_u}, $\omega_{1/6}\leq(3n)^{-6}\varepsilon_0^{-5}|\mathbf H|^2$ on $K_1\cap\{|\mathbf H|>0\}$. Hence $\omega_{1/6}$ extends by zero to a continuous function $\widetilde\omega_{1/6}$ on $K_1$. The extension is positive precisely on $K_1\cap\{|\mathbf H|>0\}\cap\{|B_I|>0\}$. On this set it is smooth, and it satisfies \eqref{eq:hypothesis_bound_evol_equation_omega} by Lemma \ref{lem:result_via_Naff_inequalities}. Thus Lemma \ref{lem:localization_lemma} applies to $\widetilde\omega_{1/6}$. Since $f\leq \frac{1}{3n}|\mathbf H|^2$, we deduce that, on $K_{1/2}\cap\{|\mathbf H|>0\}$,
\begin{equation*}
    |B_I|^2=\widetilde\omega_{1/6}^{\frac{1}{6}}f^{\frac{5}{6}}\leq \left(\frac{64(4n+9+16\gamma)}{27\beta}\right)^{\frac{1}{6}}\left(\frac{1}{3n}\right)^{\frac{5}{6}}R^{-\frac{1}{3}}|\mathbf H|^{\frac{5}{3}},
\end{equation*}
with $\beta$ and $\gamma$ as in Lemma \ref{lem:result_via_Naff_inequalities}. At points of $K_{1/2}$ where $|\mathbf H|=0$, the pinching gives $B=0$, so this inequality and the ones below hold trivially.
    To recover the full range $\sigma\in\left(0,\frac{1}{6}\right]$, we interpolate this inequality with $|B_I|^2\leq\frac{1}{3n}|\mathbf H|^2$ as follows:
    \begin{align*}
        |B_I|^2&=(|B_I|^2)^{1-6\sigma}(|B_I|^2)^{6\sigma}
        \\&\leq \left(\frac{1}{3n}|\mathbf H|^2\right)^{1-6\sigma}\left(\frac{64(4n+9+16\gamma)}{27\beta}\right)^\sigma\left(\frac{1}{3n}\right)^{5\sigma}R^{-2\sigma}|\mathbf H|^{10\sigma}
        \\&=C_1R^{-2\sigma}|\mathbf H|^{2-2\sigma},
    \end{align*}
    where
    \begin{equation*}
        C_1=C_1(n,\varepsilon_0,\sigma)\doteq \left(\frac{64(4n+9+16\gamma)}{27\beta}\right)^\sigma\left(\frac{1}{3n}\right)^{1-\sigma}>0.
    \end{equation*}
\end{proof}

\begin{rem}\label{rem:lynch_approach}
    For $\sigma\in\left(0,\frac{1}{30}\right]$, the results of Theorem \ref{thm:main_result_planarity_estimate} can also be derived from Lynch's refined estimates \cite{Lynch_PlanarityEstimate}, with neither Lemma \ref{lem:independence_from_underline_u} nor interpolation. Indeed, by Lemma $3.2$ and Proposition $5.1$ of \cite{Lynch_PlanarityEstimate} the reaction bracket in \eqref{eq:key_identity} is at least $\left(\frac{1}{2}-\sigma\right)|B_\nu|^2|B_I|^2$. Meanwhile, Proposition $5.2$ of the same paper, which relies on Lynch's sharp Kato-type inequality, bounds from below the gradient bracket with the nonnegative quantity $\frac{1}{29}|\nabla^\perp B_I|^2$. Hence, the right-hand side of \eqref{eq:key_identity} is at most $(2\sigma-1)|B_\nu|^2|B_I|^2$. This fact, together with $2(1-\sigma)\langle\zeta,\nabla\log f\rangle\le \frac{1-\sigma}{\sigma}|\zeta|^2+\sigma(1-\sigma)|\nabla\log f|^2$, turns the identity of Lemma \ref{lem:identity_log_omega} into an inequality,
    \begin{equation*}
        (\partial_t-\Delta)(\log\omega_\sigma)\leq-\frac{1-2\sigma}{\sigma}|B_\nu|^2+|\nabla\log\omega_\sigma|^2.
    \end{equation*}
    Moreover, $|B_\nu|^2\geq\frac1n|\mathbf H|^2\geq\frac{1}{n}\left(3n\varepsilon_0^{1-\sigma}\right)^{\frac{1}{\sigma}}\omega_\sigma$ by \eqref{eq:bound_Q_A_I_u}, so $\omega_\sigma$ satisfies \eqref{eq:hypothesis_bound_evol_equation_omega} with
    \begin{equation*}
        \beta=\frac{1-2\sigma}{n\sigma}\left(3n\varepsilon_0^{1-\sigma}\right)^{\frac1\sigma},\qquad\gamma=1.
    \end{equation*}
    Applying Lemma \ref{lem:localization_lemma} as above then gives the desired conclusion.  
\end{rem}

\begin{proof}[Proof of Corollary \ref{cor:lynch_planarity_results_sigma_extension}]
    To prove $(1)$, fix $(q_0,t_0)\in M\times(0,T)$ and let $R^2=t_0$. Then apply Theorem \ref{thm:main_result_planarity_estimate} with $x_0\doteq x(q_0,t_0)$ to the flow translated in time by $t_0$, i.e. $(q,t)\mapsto x(q,t+t_0)$.

    The proof of $(2)$ goes as follows. Since $M$ is closed, the flow is space-time proper. Moreover, the uniform pinching is preserved with the same $\varepsilon_0$ along the flow (see, e.g., \cite[Lemma $2.1$]{LynchNguyen_PinchedAncientSolutionsHighCodimMCF}), hence the result of the first part holds. Choose any $t_1\in(0,T)$, e.g. $t_1=T/2$. For $t\geq t_1$, the first part of the corollary implies that
    \begin{equation*}
        |B_I|^2\leq C_1t^{-\sigma}|\mathbf H|^{2-2\sigma}\leq C_1t_1^{-\sigma}|\mathbf H|^{2-2\sigma}. 
    \end{equation*}
    For $t\in[0,t_1]$, \eqref{eq:bound_Q_A_I_u} yields
    \begin{equation*}
        |B_I|^2\leq \frac{|\mathbf H|^{2\sigma}}{3n}|\mathbf H|^{2-2\sigma}\leq \frac{H_\text{max}^{2\sigma}}{3n}|\mathbf H|^{2-2\sigma},
    \end{equation*}
    where $H_\text{max}\doteq \max_{M\times[0,t_1]}|\mathbf H|$ is finite, since $M\times[0,t_1]$ is compact. Hence, on all $M\times[0,T)$, 
    \begin{equation*}
        |B_I|^2\leq C_2|\mathbf H|^{2-2\sigma},\qquad C_2\doteq \max \left\{C_1t_1^{-\sigma},\frac{H_\text{max}^{2\sigma}}{3n}\right\}.
    \end{equation*}
\end{proof}

\section{Codimension reduction for ancient solutions}\label{sec:ancient_sec_6}

The next lemma shows that, along space-time proper flows, containment in an affine subspace $y+\Pi$ propagates forward in time. The natural approach is to prove the vanishing of the squared distance $a=|\operatorname{proj}_{\Pi^\perp}(x-y)|^2$, which vanishes at the initial time and is a subsolution of the heat equation. Then, for closed flows, the maximum principle applied to $a$ gives the claim. In the noncompact case, we make use of the parabolic Omori--Yau maximum principle of \cite{Ma_ParabolicOmoriYau}. To pursue this route we have to consider the quotient $a/b$, where $b$ dominates $a$, as the function $a$ alone grows quadratically and \cite[Theorem $1.2$]{Ma_ParabolicOmoriYau} asks for sublinear growth. If, in addition, $b$ solves the heat equation, the quotient only produces gradient terms, which vanish along the parabolic Omori--Yau sequence. The idea for the right function $b$ comes from \cite[$(3.4)$]{Ecker_RegularityTheory}. One then multiplies the quotient by a negative exponential, as in \cite[Theorem $1.3$]{Ma_ParabolicOmoriYau}, to get the desired contradiction.  

\begin{lem}\label{lem:forward_confinement}
    Let $x:M^n\times[t_0,t_1]\to\mathbb R^{n+p}$, with $n\geq 2$, be a space-time proper mean curvature flow, $y\in\mathbb R^{n+p}$, and $\Pi$ a linear subspace. If $x(M,t_0)\subset y+\Pi$, then $x(M,t)\subset y+\Pi$ for every $t\in[t_0,t_1]$.
\end{lem}

\begin{proof}
    To prove that $x(M,t)\subseteq y+\Pi$ for every $t\in[t_0,t_1]$, let 
    \begin{equation*}
        a\doteq |\operatorname{proj}_{\Pi^\perp}(x-y)|^2,\qquad b\doteq 1+|x-y|^2+2n(t-t_0).
    \end{equation*}
    They satisfy $(\partial_t-\Delta)a\leq 0$ and $(\partial_t-\Delta)b=0$; compare \cite[pp. $26$ and $28$]{Ecker_RegularityTheory}. If $M$ is compact the maximum principle trivially implies $a\equiv 0$, and hence the claim. Therefore, from now on, we assume that $M$ is noncompact. Since $|\nabla b|\leq 2\sqrt b$ and $0\leq a< b$, if $\Theta\doteq a/b$,
    \begin{equation*}
        (\partial_t-\Delta)\Theta=\frac{(\partial_t-\Delta)a}{b}-\frac{a(\partial_t-\Delta)b}{b^2}+\frac{2}{b}\langle \nabla b,\nabla \Theta\rangle\leq \frac{4}{\sqrt b}|\nabla \Theta|\leq 4|\nabla \Theta|.
    \end{equation*}
    Set $\rho\doteq e^{-(t-t_0)}\Theta$. Then,
    \begin{equation}\label{eq:ineq_rho_forward_lemma}
        (\partial_t-\Delta)\rho=-\rho+e^{-(t-t_0)}(\partial_t-\Delta)\Theta\leq -\rho +4|\nabla\rho|.
    \end{equation}
    Moreover, $\rho$ is bounded, since $\Theta\in[0,1)$, and $\rho(\cdot,t_0)=0$ as $a(\cdot,t_0)=0$. Suppose $\sup\rho>0$, so that \cite[Theorem $1.2$]{Ma_ParabolicOmoriYau} applies. As a consequence, there exists a sequence of points $(q_i,t_i)$ such that
    \begin{equation*}
        \rho(q_i,t_i)\to \sup\rho,\quad |\nabla\rho|(q_i,t_i)\to 0,\quad \liminf_i(\partial_t-\Delta)\rho(q_i,t_i)\geq 0.
    \end{equation*}
    In particular, this contradicts \eqref{eq:ineq_rho_forward_lemma}, as \eqref{eq:ineq_rho_forward_lemma} evaluated along $(q_i,t_i)$ yields
    \begin{equation*}
        \limsup_i(\partial_t-\Delta)\rho(q_i,t_i)\leq -\sup\rho<0.
    \end{equation*}
    Thus $\rho\equiv 0$, which implies $a\equiv 0$, and hence $x(M,t)\subseteq y+\Pi$. 
\end{proof}

\begin{lem}\label{lem:lem_H_vanishing_or_not}
    Let $n\geq 2$, and let $x:M^n\times(-\infty,T)\to\mathbb R^{n+p}$ be a space-time proper MCF of a connected manifold $M$ without boundary that satisfies $|B|^2\leq\left(\frac{4}{3n}-\varepsilon_0\right)|\mathbf H|^2$ for some $\varepsilon_0>0$. Then, either $x(\cdot,t)$ is a static isometry onto a fixed affine $n$-plane, or $f>0$, and hence $|\mathbf H|>0$, on $M\times(-\infty,T)$. 
\end{lem}

\begin{proof}
    Since $(\partial_t-\Delta)f\geq 0$, by Lemma \ref{lem:inequality_evol_equation_f}, and $f\geq 0$, if $f(q_0,t_0)=0$ for some $(q_0,t_0)\in M\times(-\infty,T)$ the strong maximum principle implies that $f\equiv 0$ on $M\times(-\infty,t_0]$. Hence $\mathbf H\equiv 0$ and, by the pinching, $B\equiv 0$ there, so $M_{t_0}$ is contained in an affine $n$-plane $y+\Pi$. Then, Lemma \ref{lem:forward_confinement} implies that $M_t\subseteq y+\Pi$ for every $t\geq t_0$ and $B\equiv 0$ on $M\times(-\infty,T)$. Completeness from properness makes $x(\cdot,t)$ an isometry onto $y+\Pi$, thanks to \cite[Theorem $1$]{ONeill_IsometricImmersion}. 

    Otherwise $f>0$, and hence $|\mathbf H|>0$, everywhere on $M\times(-\infty,T)$. 
\end{proof}

\begin{proof}[Proof of Theorem \ref{item:codim_reduction}]
    First, consider the case $n\geq 5$. It is not restrictive to assume that $0<\varepsilon_0<\frac{1}{3n}$, since the pinching with $\varepsilon_0$ implies the pinching with any smaller constant. By Lemma \ref{lem:lem_H_vanishing_or_not}, either $\mathbf H\equiv 0$ and $x(\cdot,t)$ is a static isometry onto a fixed affine $n$-plane, or $|\mathbf H|>0$ on $M\times(-\infty,T)$, which is the condition we assume from now on.

Fix $(q_0,t_0)\in M\times(-\infty,T)$ and $R>0$. The flow $(q,t)\mapsto x(q,t+t_0)$, $t\in[-R^2,0]$, satisfies the hypothesis of Theorem \ref{thm:main_result_planarity_estimate} with $x_0\doteq x(q_0,t_0)$, which gives $|B_I|^2(q_0,t_0)\leq C_1R^{-2\sigma}|\mathbf H|^{2-2\sigma}(q_0,t_0)$. Letting $R\to \infty$ we get $|B_I|^2(q_0,t_0)=0$. Since the point $(q_0,t_0)$ is arbitrary, $B_I\equiv 0$. Then \cite[Proposition 2.5]{Naff_Planarity} gives codimension reduction, as every eigenvalue $\lambda_i$ of $h^{n+1}$ satisfies $\lambda_i^2\leq |B|^2<\frac{4}{3n}|\mathbf H|^2<|\mathbf H|^2$. Specifically, each time slice $M_t$ is contained in an $(n+1)$-dimensional affine subspace $A_t$. Actually $A_t$ is time independent. Indeed, given $s<t$, Lemma \ref{lem:forward_confinement} applied to $y+\Pi=A_s$ says that $M_t\subseteq A_s$. Let $\tilde A_t\subseteq A_t$ be the affine hull of $M_t$. If $\dim \tilde A_t=n$ then $B\equiv 0$ and $\mathbf H\equiv 0$, which contradicts the assumption $|\mathbf H|>0$. Thus $\dim\tilde A_t=n+1$ and $\tilde A_t=A_t$ by dimension. As a consequence $A_t\subseteq A_s$, which are then equal by dimension. Hence $A_t=A_s$ for every time $t$. 

Now consider the case $n\in\{2,3,4\}$. The properness assumption guarantees that each $M_t$ is complete. By \cite[Proposition $8.1$]{IRR_translators}, each $M_t$ is either flat or compact. The first case is trivial. In the latter, $|\mathbf H|>0$ by Lemma \ref{lem:lem_H_vanishing_or_not}, and \cite[Theorem 1.1]{LynchNguyen_PinchedAncientSolutionsHighCodimMCF} shows that the flow is a family of shrinking spheres. In particular, it lies in a fixed affine $(n+1)$-dimensional subspace.
\end{proof}

\section{Convexity and classification of ancient solutions}\label{sec:classification_sec_7}

We assume, throughout this section, that $x:M^n\times (-\infty,T)\to\mathbb R^{n+p}$ is as in Theorem \ref{thm:main_theorem_ancient}. We further assume that it is not a static affine $n$-plane and that $\varepsilon_0\in\left(0,\frac{1}{3n}\right)$, which is not restrictive, so that \eqref{eq:bound_Q_A_I_u} holds. By Theorem \ref{item:codim_reduction}, $H\doteq |\mathbf H|>0$ and the flow lives in an affine, fixed, $(n+1)$-dimensional subspace $\mathcal A\cong\mathbb R^{n+1}$. Moreover, $\nu$ is a smooth unit normal of $M_t$ in $\mathcal A$ that satisfies $\nabla^\perp\nu=0$. As a result $B=h\otimes \nu$, where $h\doteq h^{n+1}=\langle B,\nu\rangle$, and $(\nabla^\perp)^kB=\nabla^kh\otimes\nu$, so that $|(\nabla^\perp)^kB|=|\nabla^kh|$. We denote by $\kappa_1\leq\ldots\leq \kappa_n$ the principal curvatures, that is the eigenvalues of $h$ with respect to $g$.

\subsection{Convexity and noncollapsing}

By \cite[Proposition $9$]{AndrewsBaker_MCFPinchedSubmanifolds} and the Cauchy--Schwarz inequality, there are constants $D_1=D_1(n)$ and $D_2=D_2(n)$ such that
\begin{align}
    (\partial_t-\Delta)|\nabla h|^2&\le-2|\nabla^2 h|^2+D_1|h|^2|\nabla h|^2,\label{eq:estimate_evolution_sqnorm_gradient_B}
    \\(\partial_t-\Delta)|\nabla^2 h|^2&\le-2|\nabla^3 h|^2+D_2|h|^2|\nabla^2 h|^2+D_2|h||\nabla h|^2|\nabla^2 h|.\label{eq:estimate_evolution_sqnorm_second_cd_B}
\end{align}

\begin{prop}\label{prop:derivative_estimates_ancient}
    There exist constants $\Gamma_1=\Gamma_1(n,\varepsilon_0)$ and $\Gamma_2=\Gamma_2(n,\varepsilon_0)$ such that
    \begin{equation}\label{eq:gradient_estimates}
        |\nabla h|\leq\Gamma_1 H^2,\qquad |\nabla^2h|\leq\Gamma_2 H^3,
    \end{equation} 
    hold on $M\times(-\infty,T)$. 
\end{prop}

\begin{proof}
    Let $C$ be a positive constant depending only on $n$ and $\varepsilon_0$, which may change from line to line. The inequality $|h|^2\leq \varepsilon_0^{-1}f$, which is a consequence of the pinching condition and \eqref{eq:bound_Q_A_I_u}, will be silently used in various places of this proof. 

    Fix $(q_0,t_0)\in M\times(-\infty,T)$ and $R>0$. Let $K_\lambda$ be the sets defined in \eqref{eq:def_K_lambda} for the flow $(q,t)\mapsto x(q,t+t_0)$, $t\in[-R^2,0]$, and $x_0\doteq x(q_0,t_0)$. By space-time properness $K_1$ is compact, and $(q_0,0)\in K_{1/2}$. 

    Let $v\doteq|\nabla h|^2/f$, which, as in the proof of \cite[Theorem $4.1$]{LynchNguyen_HighCodimMCFSurgery}, satisfies $\frac{2}{f}\langle\nabla f,\nabla v\rangle\leq\frac{2}{f}|\nabla^2 h|^2$. Hence, the quotient rule, \eqref{eq:estimate_evolution_sqnorm_gradient_B} and Lemma \ref{lem:inequality_evol_equation_f} give
    \begin{equation*}
        (\partial_t-\Delta)v\leq D_1|h|^2v-2\delta_nv^2\leq Cfv-2\delta_nv^2.
    \end{equation*}
    Since $(\partial_t-\Delta)f\geq 2\delta_n fv$ by \eqref{eq:inequality_evol_f}, there exists $L_1=L_1(n,\varepsilon_0)$ such that $(\partial_t-\Delta)(v-L_1f)\leq -2\delta_nv^2$. The positive part $\omega\doteq(v-L_1f)_+$ is continuous on $K_1$ and smooth on $K_1\cap\{\omega>0\}$. On the latter set, $0<\omega\leq v$ and
    \begin{equation*}
        (\partial_t-\Delta)(\log\omega)\leq-2\delta_n \omega+|\nabla\log\omega|^2.
    \end{equation*}
    Lemma \ref{lem:localization_lemma} yields
    \begin{equation*}
        |\nabla h|^2\leq L_1f^2+\frac{32(4n+25)}{27\delta_n R^2}f\ \text{ on }K_{1/2}.
    \end{equation*}
    Letting $R\to\infty$ gives $|\nabla h|^2\leq L_1f^2\leq \frac{L_1}{9n^2}H^4$ at $(q_0,t_0)$. Because $(q_0,t_0)$ was chosen arbitrarily, this implies the first estimate of \eqref{eq:gradient_estimates} with $\Gamma_1\doteq\sqrt{L_1}/(3n)$.

    \medskip

    Now, let  $\Psi\doteq 2L_1f^2-|\nabla h|^2$ and $\tilde v\doteq |\nabla^2h|^2/\Psi$.
    By the proof of the previous estimate and \eqref{eq:bound_Q_A_I_u}, $C^{-1}H^4\leq L_1f^2\leq\Psi\leq2L_1f^2$. Since $|\nabla f|\leq CH|\nabla h|\leq CH^3$, $f>0$ and $(\partial_t-\Delta)f\geq 0$, one gets $(\partial_t-\Delta)f^2\geq -CH^6$. Together with \eqref{eq:estimate_evolution_sqnorm_gradient_B}, this gives $(\partial_t-\Delta)\Psi\geq2|\nabla^2h|^2-CH^6$. In the quotient rule for $\tilde v$, the term $\frac{2}{\Psi}\langle\nabla\Psi,\nabla\tilde v\rangle$ is absorbed by $-\frac{2}{\Psi}|\nabla^3h|^2$ as above, hence, by \eqref{eq:estimate_evolution_sqnorm_second_cd_B}, 
\begin{equation*}
    (\partial_t-\Delta)\tilde v\leq D_2|h|^2\tilde v+D_2\frac{|h||\nabla h|^2|\nabla^2h|}{\Psi}-2\tilde v^2+C\frac{\tilde v}{\Psi}H^6.
\end{equation*}
Using the pinching, the estimate $\Psi\geq C^{-1}H^4$, $|\nabla h|^2\leq CH^4$ and Young's inequality, we get
\begin{equation*}
    D_2|h|^2\tilde v+C\frac{\tilde v}{\Psi}H^6\leq CH^2\tilde v
\end{equation*}
and
\begin{equation*}
    D_2\frac{|h||\nabla h|^2|\nabla^2h|}{\Psi}=D_2\frac{|h||\nabla h|^2\sqrt{\tilde v}}{\sqrt{\Psi}}\leq CH^3\sqrt{\tilde v}\leq CH^2\tilde v+CH^4.
\end{equation*}
Thus, $(\partial_t-\Delta)\tilde v\leq -2\tilde v^2+CH^2\tilde v+CH^4$. Applying Young's inequality once again to the second term and recalling that $\varepsilon_0H^2\leq f$, we deduce that
\begin{equation*}
    (\partial_t-\Delta)\tilde v\leq -\tilde v^2+Cf^2.
\end{equation*}
    Since $(\partial_t-\Delta)f\geq 6f^2$ by \eqref{eq:inequality_evol_f}, there is $L_2=L_2(n,\varepsilon_0)$ such that $\tilde v-L_2f$ satisfies
    \begin{equation*}
        (\partial_t-\Delta)(\tilde v-L_2f)\leq-\tilde v^2\leq-(\tilde v-L_2f)^2
    \end{equation*}
    on the set where $\tilde v-L_2f>0$. Hence, Lemma \ref{lem:localization_lemma} applied to $\tilde\omega\doteq (\tilde v-L_2f)_+$ (with $\beta=\gamma=1$) gives, at $(q_0,t_0)$ and as $R\to\infty$,
    \begin{equation*}
        |\nabla^2h|^2=\tilde v\Psi\leq L_2f\cdot 2L_1f^2\leq \frac{2L_1L_2}{27n^3}H^6.
    \end{equation*}
    As $(q_0,t_0)$ is arbitrary, this yields the second inequality in \eqref{eq:gradient_estimates} with $\Gamma_2=\sqrt{(2L_1L_2)/(27n^3)}$. 
\end{proof}

\begin{thm}\label{thm:convexity_ancient}
    In the above setting, $h\geq0$ on $M\times(-\infty,T)$, namely every time slice $M_t$ has nonnegative second fundamental form in $\mathcal A$ with respect to the normal $\nu$.
\end{thm}

The proof follows White's arguments \cite{White_NatureSingularities}, as adapted by Naff \cite[Proposition $2.6$]{Naff_SingularityModels}. White obtains compactness from his regularity theory for mean convex flows, while here it comes, through the following lemma, from the derivative estimates \eqref{eq:gradient_estimates}. Naff states the lemma for complete, strictly convex, embedded and uniformly two-convex ancient solutions, but its proof uses only \eqref{eq:gradient_estimates}, $H>0$, $|h|\leq H$ and the completeness of the time slices. All of these hold in our setting. Following the notation of \cite{Naff_SingularityModels} we define, given $(q,t)\in M\times(-\infty,T)$,
\begin{equation*}
    P(q,t,r,\theta)\doteq B_{g(t)}(q,r)\times[t-\theta,t],\qquad  \hat P(q,t,L,\theta)\doteq P\left(q,t,\frac{n-1}{H(q,t)}L,\left(\frac{n-1}{H(q,t)}\right)^2\theta\right).
\end{equation*}

\begin{lem}[{\cite[Lemmas $2.4$ and $2.5$]{Naff_SingularityModels}}]\label{lem:canonical_neighbourhood}
    There exists $\hat r=\hat r(n,\varepsilon_0)>0$ such that for every space-time point $(q_0,t_0)$ and every $(q,t)\in\hat P(q_0,t_0,\hat r,\hat r^2)$ we have 
    \begin{equation*}
        \frac{1}{4}H(q_0,t_0)\leq H(q,t)\leq 4H(q_0,t_0).
    \end{equation*}
    Moreover, for all nonnegative integers $k,\ell$ there exists $C_{k,\ell}=C_{k,\ell}(n,\varepsilon_0)<\infty$ such that
    \begin{equation*}
        \left|\partial_t^\ell\nabla^k h\right|\leq C_{k,\ell}H^{2\ell+k+1}
    \end{equation*}
    holds on $M\times(-\infty,T)$. 
\end{lem}

\begin{proof}[Proof of Theorem \ref{thm:convexity_ancient}]
    Let
    \begin{equation*}
        \kappa_*\doteq\inf_{M\times(-\infty,T)}\frac{\kappa_1}{H},
    \end{equation*}
    which is finite because of the pinching condition, and choose a sequence $(q_j,t_j)$ such that $\frac{\kappa_1}{H}(q_j,t_j)\to\kappa_*$. Let $Q_j\doteq H(q_j,t_j)^{-1}$, and consider the parabolic rescaling
    \begin{equation*}
        F_j(q,\tau)\doteq Q_j^{-1}\left(x(q,t_j+Q_j^2\tau)-x(q_j,t_j)\right).
    \end{equation*}
    The neighbourhood $\hat P(q_j,t_j,\hat r,\hat r^2)$ becomes $P_j\doteq B_{g_j(0)}(q_j,(n-1)\hat r)\times[-(n-1)^2\hat r^2,0]$, where $g_j$ is the metric induced by $F_j$.
    Each $F_j$ is a mean curvature flow in $\mathcal A\cong \mathbb R^{n+1}$ with $H_j(q_j,0)=1$, and Lemma \ref{lem:canonical_neighbourhood} gives, through the estimate $\frac{1}{4}\leq H_j\leq 4$ on $P_j$, that
    \begin{equation*}
        \sup_{P_j}\left|\partial_\tau^\ell\nabla^kh_j\right|\leq4^{2\ell+k+1}C_{k,\ell}.
    \end{equation*}

    As in the proof of \cite[Proposition $2.6$]{Naff_SingularityModels}, standard compactness results give a subsequence converging smoothly on compact subsets to a limit flow $F_\infty:U\times(-T_\infty,0]\to\mathbb R^{n+1}$; see \cite{Cooper_CompactnessThm}. Here $U\subset\mathbb R^n$ is a connected open set, $T_\infty>0$, and $q_j\to q_\infty\in U$. Since the ratio $\kappa_1/H$ is scale-invariant and depends continuously on $g$ and $h$, the convergence gives, on $U\times(-T_\infty,0]$,
    \begin{equation*}
        H_\infty\geq\frac14,\qquad\frac{\kappa_{1,\infty}}{H_\infty}\geq\kappa_*,\qquad\frac{\kappa_{1,\infty}}{H_\infty}(q_\infty,0)=\kappa_*.
    \end{equation*}
    So $\kappa_{1,\infty}/H_\infty$ attains its minimum over $U\times(-T_\infty,0]$ on the final time slice. Therefore \cite[Proposition A$4$]{White_NatureSingularities} applies, showing that $\kappa_{1,\infty}/H_\infty$ is a nonnegative constant near $(q_\infty,0)$; see also \cite[Proposition A$3$]{HaslhoferKleiner_ParabolicBlowDown}. In particular $\kappa_*=\frac{\kappa_{1,\infty}}{H_\infty}(q_\infty,0)\geq0$, that is $h\geq0$ on $M\times(-\infty,T)$.
\end{proof} 

\begin{proof}[Proof of Theorem \ref{item:embedding_noncollapsness}]
    By Theorem \ref{thm:convexity_ancient}, every $M_t$ has nonnegative sectional curvature. Also, $\kappa_{n-1}\kappa_n>0$ everywhere. Otherwise, $\kappa_1=\ldots=\kappa_{n-1}=0$ at some point, and there $|h|^2=H^2$, which violates the pinching hypothesis. Moreover, $M$ is orientable, because $\nu$ is a global unit normal. Hence, Sacksteder's theorems prove that $x(\cdot,t)$ is an embedding onto the boundary of a convex body $\Omega_t\subseteq \mathcal A$; see \cite{Sacksteder_HypersurfacesNoNegSectionalCurvatures,Wu_Convex}.   

    As noticed in the Introduction, the flow is uniformly $(n-1)$-convex. Hence, for each $t_1<T$, \cite[Theorem $1.1$]{BourniLangfordLynch_Noncollapsing} applied to the flow $t\mapsto\partial\Omega_{t+t_1}$, $t\in(-\infty,0]$, shows that the inscribed radius of $\Omega_t$ is at least $H^{-1}$ at every point of $\partial\Omega_t$. Since $t_1$ is arbitrary, this proves Theorem \ref{item:embedding_noncollapsness}.
\end{proof}

\subsection{Classification}

\begin{proof}[Proof of Theorem \ref{item:classification}]
    The flow is noncollapsed, so, by Theorems $1.2$ and $1.3$ of \cite{BourniLangfordLynch_Noncollapsing}, its blow-down is not a hyperplane of multiplicity two. Not even a blow-down equal to a hyperplane of multiplicity one would be possible, as, by \cite[Lemma $2.1$]{BourniLangfordTinaglia_ConvexSolutions}, the flow would be a static hyperplane, against $H>0$. So, by the same lemma, after a fixed rotation the parabolic blow-downs converge locally smoothly to a multiplicity-one shrinking $\mathbb R^k\times S^{n-k}$, with $0\leq k\leq n-1$. 
The pinching \eqref{eq:pinching} is scale invariant, so it passes to smooth limits. On $\mathbb R^k\times S^{n-k}$ it gives $4k<n$.
    The flow is smooth and properly embedded in $\mathcal A$, so it is a unit-regular, integral Brakke flow, which is asymptotically $(n,k)$-cylindrical in the sense of 
\cite[Definition $1.1$]{BamlerLai_Classification}. The classification follows from Theorem $1.2$ of the same paper, where flying wings occur only when $k\geq 2$. 
\end{proof}

\bibliographystyle{amsalpha}
\bibliography{bibliography}

\end{document}